\documentclass{article}
\usepackage{amssymb,amsfonts,amsmath,amsthm}
\usepackage{hyperref}
\usepackage{parskip}
\usepackage{setspace}
\usepackage{fancyhdr}
\usepackage[dvipsnames]{xcolor}
\usepackage[numbers]{natbib}
\usepackage[bottom]{footmisc}
\usepackage[utf8]{inputenc}
\usepackage[english]{babel}
\usepackage{xcolor}
\usepackage{graphicx}
\usepackage{geometry}
\newcommand\myshade{85}
\colorlet{mylinkcolor}{blue}
\colorlet{mycitecolor}{red}
\colorlet{myurlcolor}{Aquamarine}
\hypersetup{
	linkcolor  = mylinkcolor,
	citecolor  = mycitecolor,
	urlcolor   = myurlcolor!\myshade!black,
	colorlinks = true,
}
\title{Type-(I,II) Interpolations and some asymptotic expansions using Ramanujan's master theorem}

\author{Omprakash Atale${}^{1}$\footnote{${}^{1}$ Department of Mathematics, 
Khandesh College Education Society's. Moolji Jaitha College, India.}} 
\date{July 1, 2022}

\begin{document}

\maketitle

\begin{abstract}
The theory of Mellin transform is an incredibly useful tool in evaluating some of the well known results for the zeta function. Ramanujan in his quarterly reports \cite{1} gave a theorem for Mellin transform which is now known as Ramanujan's master theorem \cite{2}. In this paper, we have derived some extended versions of Ramanujan's master theorem based on our previous results \cite{3} and applied them to some special functions such as known as the Riesz function $R(z)$ and generalized binomial function. Some asymptotic expansions using extended Ramanujan's master theorem are also derived.
  
\end{abstract}
\newtheorem{theorem}{Theorem}
\newtheorem{corollary}{Corollary}[theorem]
\newtheorem{lemma}[theorem]{Lemma}
\theoremstyle{definition}
\newtheorem{definition}{Definition}[section]
\renewcommand\qedsymbol{$\blacksquare$}

\section{Introduction}
\textbf{1.1.} In \cite{4}, Riesz defined the following function of a complex variable $z$ 
\begin{equation}
R\left( z \right) := \sum\limits_{n = 1}^\infty  {\frac{{{{\left( { - 1} \right)}^{n + 1}}}}{{\left( {n - 1} \right)!\zeta \left( {2n} \right)}}{z^n}}. \end{equation}
The above functions is now known as the Riesz function. The radius of convergence of $R(z)$ is infinite because as $n\to\infty$, $\zeta(2n)\to 1$. Riesz conjectured that the Riemann hypothesis is equivalent to the condition that $\forall \epsilon > 0$ as $x\to\infty$ in $\mathbb{R}$, $R(x){ \ll _\varepsilon }{x^{\frac{1}{4} + \varepsilon }}$. This is known as the Riesz criterion.

\textbf{1.2.} Hardy and littlewood \cite{5} published similar series equivalence for the Riemann hypothesis by defining the function
\begin{equation}
H\left( x \right): = \sum\limits_{n = 1}^\infty  {\frac{{{{\left( { - 1} \right)}^n}}}{{n!\zeta \left( {2n + 1} \right)}}{x^n}}.     
\end{equation}
They conjectured that the Riemann hypothesis is equivalent to the condition that $\forall \delta > 0$ as $x\to\infty$, $H(x)={\rm O}\left( {{x^{ - \frac{1}{4} + \delta }}} \right)$. 

\textbf{1.3.} Ramanujan in his quarterly reports \cite{1} derived the following theorem. If $f$ has expansion of the form
\begin{equation}
f(x)=\sum_{n=0}^{\infty}(-1)^{n} \frac{\phi(n)}{n !} x^{n} \end{equation}
where $\phi(n)$ has a natural and continuous extension such that $\phi(0) \neq 0$, then for $n>0$, we have
\begin{equation}
\int_{0}^{\infty} x^{n-1}\left(\sum_{n=0}^{\infty}(-1)^{n} \frac{\phi(n)}{n !} x^{n}\right) dx =\phi(-n)\Gamma(n),   
\end{equation}
where $n$ is any positive integer. Ramanujan's method for deriving his master's theorem was unconventional and his theorem had a problem with convergence of the integral. Hardy established some boundaries to the value of $\phi$ and derived a theorem that is in all respect convergent. Following is Hardy's version of the above theorem. Let $\varphi(z)$ be an analytic (single-valued) function, defined on a half-plane $H(\delta)=\{z \in {C}: \Re (z) \geq-\delta\}$ for some $0<\delta<1 .$ Suppose that, for some $A<\pi, \phi$ satisfies the growth condition $|\phi(v+i w)|<C e^{P v+A|w|} \forall z=v+i w \in H(\delta)$. Let $0<x<e^{-p}$ the growth condition shows that the series $\Phi(x)=\phi(0)-x \phi(1)+x^{2} \phi(3) \ldots$ converges. The residue theorem yields
\begin{equation}
\Phi(x)=\frac{1}{2 \pi i} \int_{c-i \infty}^{c+i \infty} \frac{\pi}{\sin s \pi} \phi(-s) x^{-s} d s
\end{equation}
for any $0<c<\delta$. Observe that $\pi / \sin \pi s$ has poles at $s=-n$ for $n=0,1,2 \ldots$ with residue $(-1)^{n}$. The above integral converges absolutely and uniformly for $c \in(a, b)$ and $0<a<b<\delta$. Using Mellin inversion formula, $\forall 0<\Re s<\delta$, we get
\begin{equation}
\int_{0}^{\infty} x^{s-1}\left\{\phi(0)-x \phi(1)+x^{2} \phi(3) \ldots\right\} d x=\frac{\pi}{\sin s \pi} \phi(-s). 
\end{equation}
The substitution $\phi(u) \rightarrow \phi(u) / \mathrm{\Gamma}(u+1)$ in Eqn. (6) establishes Ramanujan's master theorem in its original form (Eqn. (4)).
\section{The function $\bar{R}(z)$}
\textbf{2.1.} Define the function $\bar R(z)$ for $z\in\mathbb{C}$
\begin{equation}
\bar R\left( z \right) = \sum\limits_{n = 1}^\infty  {R\left( {\frac{z}{{{n^2}}}} \right)}     
\end{equation}
Replacing $R(z)$ in above sum with Eqn. (1) by interchanging the order of summation and using pointwise absolute convergence, it can be shown that $\bar R(z)=ze^{-z}$.   Riesz \cite{4} proved that
\begin{equation}
\int\limits_0^\infty  {{x^{ - \frac{s}{2} - 1}}} \sum\limits_{n = 1}^\infty  {R\left( {\frac{x}{{{n^2}}}} \right)} dx = \frac{{\Gamma \left( {1 - \frac{1}{2}s} \right)}}{{\zeta \left( s \right)}}    
\end{equation}
Now let $s\in\mathbb{C}$ and $S:=[1+\delta, 2-\delta], 0<\delta<\frac{1}{2}$. Then, for $x\in(0,1]$ using $|R(x)|\ll x$ we have for $\sigma\in S$, $\left| {{x^{ - \frac{s}{2} - 1}}R\left( x \right)} \right| \ll {x^{ - \frac{\sigma }{2}}} \le {x^{ - \frac{1}{s} - \frac{\delta }{2}}}$ which converges if $\sigma=\Re s\in S$. Since $\delta<1$, for $x\in(0,\infty]$ we have $\left| {{x^{ - \frac{s}{2} - 1}}R\left( x \right)} \right| \ll {x^{ - \frac{\sigma }{2} - 1 + \frac{1}{2} + \varepsilon }} \le {x^{ - 1 - \left( {\frac{\delta }{2} - \varepsilon } \right)}}$ which converges for $\delta>2\epsilon$. Choosing $\delta$ and $\epsilon$ such that $2\epsilon<\delta<\frac{1}{2}$ the integral converges absolutely for $s\in\mathbb{C}$ such that $\sigma\in S$. Eqn. (8) can also be written as 
\begin{equation}
\int\limits_0^\infty  {{x^{s - 1}}} \sum\limits_{n = 1}^\infty  {R\left( {\frac{x}{{{n^2}}}} \right)} dx = \frac{{s\Gamma \left( s \right)}}{{\zeta \left( { - 2s} \right)}}    
\end{equation}
provided that $s$ is not a positive integer. If we compare Eqn. (9) with Eqn. (4) we get $f\left( x \right) \equiv \bar R\left( x \right): = \sum\limits_{n = 1}^\infty  {R\left( {\frac{x}{{{n^2}}}} \right)}$ and $\phi \left( s \right) = \frac{{ - s}}{{\zeta \left( {2s} \right)}}$. Now, using Eqn. (3) we can construct the following series representation for $\bar R(x)$
\begin{equation}
\bar{R}\left( x \right) = \sum\limits_{n = 0}^\infty  {{{\left( { - 1} \right)}^{n + 1}}} \frac{n}{n!{\zeta \left( {2n} \right)}}{x^n}.    
\end{equation}
\section{On extended Ramanujan's master theorem }
\textbf{3.1.} Using Eqn. (4), Ramanujan proved some theorems on asymptotic expansions which are now recorded in B. Berndt's 2$^{nd}$ Vol. [\cite{6}, Pg. 306]. In this section, we are going to prove similar asymptotic expansions for some extended versions of Ramanujan's master theorem.

\textbf{3.2.} In \cite{3}, we proved that for $\Re{s}>0$\newline
(i)
\begin{equation}
\int\limits_0^\infty  {{x^{s - 1}}\sum\limits_{n = 0}^\infty  {\phi \left( {2n + 1} \right)\frac{{{{\left( { - 1} \right)}^{n}}}}{{\left( {2n + 1} \right)!}}{x^{2n + 1}}dx = \phi \left( { - s} \right)\Gamma \left( s \right)\sin \frac{{\pi s}}{2}} },
\end{equation}
(ii)
\begin{equation}
\int\limits_0^\infty  {{x^{s - 1}}\sum\limits_{n = 0}^\infty  {\phi \left( {2n} \right)\frac{{{{\left( { - 1} \right)}^{n}}}}{{\left( {2n} \right)!}}{x^{2n}}dx = \phi \left( { - s} \right)\Gamma \left( s \right)\cos \frac{{\pi s}}{2}} }.
\end{equation}
Proof of Eqn. (7) and (8) is fairly simple and can be derived from the the Mellin transform of $\sin{ax}$ and $\cos{ax}$ respectively in a similar way as Ramanujan derived his master's theorem from the gamma function which is the Mellin transform of $e^{-x}$. Eqn. (12) is also derived by Ramanujan himself in his quarterly reports [\cite{1}, Pg. 318, Corollary (i)] but his proof is quite different than the one we have presented. There is no account of Eqn. (11) is his work. Now, replace $\phi(2n+1)$ with $\frac{\phi(2n+1)}{(-1)^{n}}$ in Eqn. (11) and $\phi(2n)$ with $\frac{\phi(2n)}{(-1)^{n}}$ in Eqn. (12) to get
\begin{equation}
\int\limits_0^\infty  {{x^{s - 1}}\sum\limits_{n = 0}^\infty  {\frac{{\phi \left( {2n + 1} \right)}}{{\left( {2n + 1} \right)!}}{x^{2n + 1}}dx = {{\left( { - 1} \right)}^{\frac{{s + 1}}{2}}}} } \phi \left( { - s} \right)\Gamma \left( s \right)\sin \frac{{\pi s}}{2}    
\end{equation}
and 
\begin{equation}
\int\limits_0^\infty  {{x^{s - 1}}\sum\limits_{n = 0}^\infty  {\frac{{\phi \left( {2n} \right)}}{{\left( {2n} \right)!}}{x^{2n}}dx = {{\left( { - 1} \right)}^{\frac{s}{2}}}} } \phi \left( { - s} \right)\Gamma \left( s \right)\cos \frac{{\pi s}}{2}    
\end{equation}
respectively. Using some basic properties of alternating series and definition of $f(x)$ from Eqn. (3), we get
\begin{equation}
\frac{1}{2}\left[ {f\left( { - x} \right) - f\left( x \right)} \right] = \sum\limits_{n = 0}^\infty  {\frac{{\phi \left( {2n + 1} \right)}}{{\left( {2n + 1} \right)!}}{x^{2n + 1}}}.     
\end{equation}
and 
\begin{equation}
\frac{1}{2}\left[ {f\left( x \right) + f\left( { - x} \right)} \right] = \sum\limits_{n = 0}^\infty  {\frac{{\phi \left( {2n} \right)}}{{\left( {2n} \right)!}}{x^{2n}}}.     
\end{equation}
Thus, we can now deduce the following theorem for Eqn. (11) and (12) 
\begin{theorem}
For $\Re{s}>0$, we have
\begin{equation}
\int\limits_0^\infty  {{x^{s - 1}}} \left[ {f\left( { - x} \right) - f\left( x \right)} \right]dx = 2{\left( { - 1} \right)^{\frac{{s + 1}}{2}}}\phi \left( { - s} \right)\Gamma \left( s \right)\sin \frac{{\pi s}}{2}    
\end{equation}
and 
\begin{equation}
\int\limits_0^\infty  {{x^{s - 1}}} \left[ {f\left( x \right) + f\left( { - x} \right)} \right]dx = 2{\left( { - 1} \right)^{\frac{s}{2}}}\phi \left( { - s} \right)\Gamma \left( s \right)\cos \frac{{\pi s}}{2}.    
\end{equation}
where $f(x)$ is defined by Eqn. (3) and is continuous  on $\Omega=[-\infty,\infty]$. $\phi(s)$ is any function which has a natural and continuous extension such that $\phi(0)\ne 0$.
\end{theorem}

Throughout the sequel, we will refer to Eqn. (17) as \textbf{type I interpolation} and Eqn. (18) as \textbf{type II interpolation}. If we subtract Eqn. (18) from Eqn. (17), we get original version of Ramanujan's master theorem.
\begin{align}
\int\limits_0^\infty  {{x^{s - 1}}} f\left( x \right)dx &= {\left( { - 1} \right)^{\frac{s}{2}}}\phi \left( { - s} \right)\Gamma \left( s \right)\left( {\cos \frac{{\pi s}}{2} - i\sin \frac{{\pi s}}{2}} \right) \\&= {\left( { - 1} \right)^{\frac{s}{2}}}\phi \left( { - s} \right)\Gamma \left( s \right){e^{ - \frac{{i\pi s}}{2}}} \\&= \phi \left( { - s} \right)\Gamma \left( s \right).    
\end{align} 
To write Eqn. (17) and (18) in Hardy's form, replace $\phi(n)$ with $\phi \left( n \right)\Gamma \left( {n + 1} \right)$ and we get
\begin{align}
\int\limits_0^\infty  {{x^{s - 1}}} \left( {\sum\limits_{n = 0}^\infty  {\phi \left( n \right)} {x^n} - \sum\limits_{m = 0}^\infty  {\phi \left( m \right)} {{\left( { - x} \right)}^m}} \right)dx &= 2{\left( { - 1} \right)^{\frac{{s + 1}}{2}}}\Gamma \left( {1 - s} \right)\Gamma \left( s \right)\sin \frac{{\pi s}}{2} \\&= 2{\left( { - 1} \right)^{\frac{{s + 1}}{2}}}\frac{\pi }{{\sin \pi s}}\sin \frac{{\pi s}}{2}  
\end{align}
and
\begin{align}
\int\limits_0^\infty  {{x^{s - 1}}} \left( {\sum\limits_{n = 0}^\infty  {\phi \left( n \right)} {{\left( { - x} \right)}^n} + \sum\limits_{m = 0}^\infty  {\phi \left( m \right)} {x^m}} \right)dx &= 2{\left( { - 1} \right)^{\frac{s}{2}}}\Gamma \left( {1 - s} \right)\Gamma \left( s \right)\cos \frac{{\pi s}}{2} \\&= 2{\left( { - 1} \right)^{\frac{s}{2}}}\frac{\pi }{{\sin \pi s}}\cos \frac{{\pi s}}{2}.  
\end{align}
For Eqn. (23) and (25), we have to subject $\phi(n)$ and $s$ to the conditions presented in \textsection 1.3 in Hardy's formulation of Ramanujan's master theorem. That is, $0<\Re{s}<1$ and $|\phi(v+i w)|<C e^{P v+A|w|} \forall z=v+i w \in H(\delta)$.

\textbf{3.3.} As an example, consider the following binomial theorem for $a>0$:
\begin{equation}
{\left( {1 + x} \right)^{ - a}} = \sum\limits_{n = 0}^\infty  {\frac{{\Gamma \left( {n + a} \right)}}{{\Gamma \left( a \right)}}} \frac{{{{\left( { - x} \right)}^m}}}{{m!}}    
\end{equation}
from which we get $\phi \left( s \right) = \frac{{\Gamma \left( {s + a} \right)}}{{\Gamma \left( a \right)}}$. Applying type I interpolation and type II interpolation yields
\begin{equation}
\int\limits_0^\infty  {{x^{s - 1}}} \left( {\frac{1}{{{{\left( {1 - x} \right)}^a}}} - \frac{1}{{{{\left( {1 + x} \right)}^a}}}} \right)dx = 2{\left( { - 1} \right)^{\frac{{s + 1}}{2}}}\frac{{\Gamma \left( {a - s} \right)\Gamma \left( s \right)}}{{\Gamma \left( a \right)}}\sin \frac{{\pi s}}{2}    
\end{equation}
and 
\begin{equation}
\int\limits_0^\infty  {{x^{s - 1}}} \left( {\frac{1}{{{{\left( {1 + x} \right)}^a}}} + \frac{1}{{{{\left( {1 - x} \right)}^a}}}} \right)dx = 2{\left( { - 1} \right)^{\frac{s}{2}}}\frac{{\Gamma \left( {a - s} \right)\Gamma \left( s \right)}}{{\Gamma \left( a \right)}}\cos \frac{{\pi s}}{2}    
\end{equation}
valid for $\Re{s}>0$.

\textbf{3.4.} Recall from Eqn. (10) that we can write $\bar{R}(x)$ as 
\begin{equation}
\bar{R}\left( x \right) = \sum\limits_{n = 0}^\infty  {{{\left( { - 1} \right)}^{n }}} \frac{-n}{n!{\zeta \left( {2n} \right)}}{x^n}.    
\end{equation}
which gives $\phi(n)=\frac{-n}{\zeta(2n)}$. Applying type I interpolation and type II interpolation for $s\in\mathbb{C}$ on Eqn. (29) gives 
\begin{equation}
\int\limits_0^\infty  {{x^{ - \frac{s}{2} - 1}}\left[ {\bar R\left( { - x} \right) - \bar R\left( x \right)} \right]} dx = 2{\left( { - 1} \right)^{\frac{{10 - s}}{4}}}\frac{{\Gamma \left( {1 - \frac{1}{2}s} \right)}}{{\zeta \left( s \right)}}\sin \frac{{\pi s}}{4} \end{equation}
and
\begin{equation}
\int\limits_0^\infty  {{x^{ - \frac{s}{2} - 1}}\left[ {\bar R\left( x \right) + \bar R\left( { - x} \right)} \right]} dx = 2{\left( { - 1} \right)^{1 - \frac{s}{4}}}\frac{{\Gamma \left( {1 - \frac{1}{2}s} \right)}}{{\zeta \left( s \right)}}\cos \frac{{\pi s}}{4}.   
\end{equation}
From Eqn. (18), we get the following corollary.
\begin{corollary}
Let $\Re{s}>0$ and $f(x)$ be a continuous on $\Omega=[-\infty,\infty]$. Then
\begin{equation}
\int\limits_{-\infty}^\infty  {{x^{s - 1}}} f(x)dx = 2{\left( { - 1} \right)^{\frac{s}{2}}}\phi \left( { - s} \right)\Gamma \left( s \right)\cos \frac{{\pi s}}{2}.    
\end{equation}
Here, $\phi(s)$ is any function which has a natural and continuous extension such that $\phi(0)\ne 0$
\end{corollary}

\section{ On some asymptotic expansions}
\textbf{4.1.} From [\cite{7}, Sec. 17][\cite{8}, P. 196], we have the following estimate
\begin{equation}
\sum\limits_{k = 1}^n {{k^s}}  = \zeta \left( { - s} \right) + \frac{{{{\left( {n + \frac{1}{2}} \right)}^{s + 1}}}}{{s + 1}} + {\rm O}\left( {{n^{s - 1}}} \right)    
\end{equation}
for any complex number $s$ and positive integer $n$. Using this estimate, we are going to derive an extended version of Ramanujan's master theorem though which we can derive some asymptotic expansions. First, we will present the main result of this theorem in the form of theorem 1 and then applications will follow. 

\begin{theorem}
For $\Re{s}>0$, we have
\begin{equation}
\sum\limits_{k = 1}^n {\int\limits_0^\infty  {{x^{s - 1}}f\left( {kx} \right)dx}  = \phi \left( { - s} \right)} \Gamma \left( s \right)\left( {\zeta \left( s \right) + \frac{{{{\left( {n + \frac{1}{2}} \right)}^{1 - s}}}}{{1 - s}} + {\rm O}\left( {{n^{ - s - 1}}} \right)} \right)    
\end{equation}
where 
\begin{equation}
f\left( x \right) = \sum\limits_{v = 0}^\infty  {\frac{{\phi \left( v \right){{\left( { - x} \right)}^v}}}{{v!}}}     
\end{equation}
and $\phi(0)\neq{0}$.
\end{theorem}
\begin{proof}
Replace $s$ with $-s$ in Eqn. (32) and use the definition of gamma function to get
\begin{equation}
\sum\limits_{k = 1}^n {\int\limits_0^\infty  {{x^{s - 1}}{e^{ - kx}}dx}  = } \Gamma \left( s \right)\left( {\zeta \left( s \right) + \frac{{{{\left( {n + \frac{1}{2}} \right)}^{1 - s}}}}{{1 - s}} + {\rm O}\left( {{n^{ - s - 1}}} \right)} \right).    
\end{equation}
Now, replace $x$ with $mx$ for $m>0$ and then $m$ with $u^v$ where $u>0$ and $0\leq{k}<\infty$. Further, multiply both sides by $\frac{{{f^{\left( v \right)}}\left( a \right){h^v}}}{{v!}}$ and sum on $v$ from $0$ to $\infty$. Here, $f$ shall be specified later. 
\begin{align}
\sum\limits_{v = 0}^\infty  {\frac{{{f^{\left( v \right)}}\left( a \right){h^v}}}{{v!}}} & \sum\limits_{k = 1}^n {\int\limits_0^\infty  {{x^{s - 1}}{e^{ - k{u^v}x}}dx}}  \\&{{= \sum\limits_{v = 0}^\infty  {\frac{{{f^{\left( v \right)}}\left( a \right){{\left( {h{u^{ - s}}} \right)}^v}}}{{v!}}} } \Gamma \left( s \right)\left( {\zeta \left( s \right) + \frac{{{{\left( {n + \frac{1}{2}} \right)}^{1 - s}}}}{{1 - s}} + {\rm O}\left( {{n^{ - s - 1}}} \right)} \right)}.   
\end{align}
Expand ${{e^{ - k{u^v}x}}}$ in it's Maclaurin series and invert the order of summation and integration to get \cite{1}
\begin{equation}
\sum\limits_{k = 1}^n {\int\limits_0^\infty  {{x^{s - 1}}\sum\limits_{v = 0}^\infty  {\frac{{f\left( {h{u^v} + a} \right){{\left( { - kx} \right)}^v}}}{{v!}}} dx}  = f\left( {h{u^{ - s}} + a} \right)} \Gamma \left( s \right)\left( {\zeta \left( s \right) + \frac{{{{\left( {n + \frac{1}{2}} \right)}^{1 - s}}}}{{1 - s}} + {\rm O}\left( {{n^{ - s - 1}}} \right)} \right).
\end{equation}
Now, define $f\left( {h{u^{ - s}} + a} \right) = \phi \left( { - s} \right)$ where $a, h$ and $u$ are constants and we get
\begin{equation}
\sum\limits_{k = 1}^n {\int\limits_0^\infty  {{x^{s - 1}}\sum\limits_{v = 0}^\infty  {\frac{{\phi \left( v \right){{\left( { - kx} \right)}^v}}}{{v!}}} dx}  = \phi \left( { - s} \right)} \Gamma \left( s \right)\left( {\zeta \left( s \right) + \frac{{{{\left( {n + \frac{1}{2}} \right)}^{1 - s}}}}{{1 - s}} + {\rm O}\left( {{n^{ - s - 1}}} \right)} \right).   \end{equation}
This complete our proof.
\end{proof}
For simplicity purpose, we will use the following notation throughout the sequel:
\begin{equation}
{\Delta _n}\left( s \right) = \zeta \left( s \right) + \frac{{{{\left( {n + \frac{1}{2}} \right)}^{1 - s}}}}{{1 - s}} + {\rm O}\left( {{n^{- s - 1}}} \right).    
\end{equation}
Furthermore, $z^{-}$ denotes the set of non-positive integers, $\sigma=\Re{s}$, $\sum\nolimits_k^* {} $ denotes that those values of $k$ giving undefined summands are excluded from the summation.
\begin{theorem}
Let $p>0$ and $m\in\mathbb{C}$ and define
\begin{equation}
f\left( x \right) = \sum\limits_{v = 1}^\infty  {{e^{ - x{v^p}}}{v^{m - 1}}}     
\end{equation}
Then, as $x\to 0^{+}$
\begin{equation}
\sum\limits_{k = 1}^n {f\left( {kx} \right)}  \sim \frac{{\Gamma \left( {\frac{m}{p}} \right){\Delta _n}\left( {\frac{m}{p}} \right)}}{{p{x^{\frac{m}{p}}}}} + \sum\limits_{l = 0}^\infty  {\zeta \left( {1 - m + pl} \right)} {\Delta _n}\left( { - l} \right)\frac{{{{\left( { - x} \right)}^l}}}{{l!}} + \frac{{\zeta \left( {1 - m + p} \right)}}{x}\left( {2 + {\rm O}\left( n^{-2} \right)} \right)    
\end{equation}
if $\frac{m}{p}\notin{Z^{-}}$. And if $\frac{m}{p}=-r\in{Z^{-}}$, then
\begin{equation*}
\sum\limits_{k = 1}^n {f\left( {kx} \right)}  \sim \frac{{{\Delta _n}\left( { - r} \right)}}{p}\left\{ {\left( {{H_r} - \gamma } \right) + p\gamma  - \log x} \right\}\frac{{{{\left( { - x} \right)}^r}}}{{r!}} + {\Delta _n}\left( { - r} \right)\sum\limits_{l = 0}^\infty  {^*\zeta \left( {1 - m + pl} \right)} \frac{{{{\left( { - x} \right)}^l}}}{{l!}}    
\end{equation*}
\begin{equation}
 + \frac{{\zeta \left( {1 - m + p} \right)}}{x}\left( {2 + {\rm O}\left( {{n^{- 2}}} \right)} \right)    
\end{equation}
where $\gamma$ is Euler's constant and ${{H_r} = \sum\limits_{k = 1}^r {\frac{1}{k}} }$.

\end{theorem}
\begin{proof}
From Eqn. (33) and (42) it can be easily seen that
\begin{equation}
\sum\limits_{k = 1}^n {\int\limits_0^\infty  {{x^{s - 1}}f\left( {kx} \right)} } dx = \Gamma \left( s \right)\zeta \left( {1 - m + ps} \right){\Delta _n}\left( s \right)    
\end{equation}
provided that $\sigma>{\sup \left\{ {0,\frac{m}{p}} \right\}}$. Using Mellin inversion  formula, we get
\begin{equation}
\sum\limits_{k = 1}^n {f\left( {kx} \right)}  = \frac{1}{{2\pi i}}\int\limits_{a - i\infty }^{a + i\infty } {\Gamma \left( s \right)\zeta \left( {1 - m + ps} \right){\Delta _n}\left( s \right){x^{ - s}}} ds    
\end{equation}
where $a>{\sup \left\{ {0,\frac{m}{p}} \right\}}$. Now, consider
\begin{equation}
{I_{M,T}} = \frac{1}{{2\pi i}}\int\limits_{{\Omega _{M,T}}} {\Gamma \left( s \right)\zeta \left( {1 - m + ps} \right){\Delta _n}\left( s \right){x^{ - s}}ds}     
\end{equation}
where ${{\Omega _{M,T}}}$ is a rectangle with vertices $a\pm iT$ and $-M\pm iT$, where $T>0$, $M=T+\frac{1}{2}$ and $N>\frac{|m|}{p}$.For $\frac{m}{p}\notin{Z^{-}}$, the integrand has simple poles at $s=\frac{m}{p}$ and $s=0,-1,-2,...,-N$ on the interior of ${{\Omega _{M,T}}}$. When $\frac{m}{p}=-r\in{Z^{-}}$ then there is a double pole at $s=-r$. Furthermore, for both cases of $\frac{m}{p}$, there is a pole at $s=1$ due to ${\Delta _n}\left( s \right)$. By residue theorem and using the expansions
\begin{equation}
\Gamma \left( s \right) = \frac{{{{\left( { - 1} \right)}^r}}}{{r!\left( {s + r} \right)}} + \frac{{{{\left( { - 1} \right)}^r}}}{{r!}}\left( {{H_r} - \gamma } \right) + ...,0 < \left| {s + r} \right| < 1    
\end{equation}
\begin{equation}
\zeta \left( {1 - m + ps} \right) = \frac{1}{{p\left( {s + r} \right)}} + \gamma  + ...    
\end{equation}
and
\begin{equation}
{x^{ - s}} = {x^r} - {x^r}\log x\left( {s + r} \right) + ...   
\end{equation}
 we get the value of residue where the summand in Eqn. (43) and (44) has $N$ as the upper limit of summation instead of infinity. Using Stirling's formula and Eqn. (3.12) from \cite{6}, it can be shown that the integrand of Eqn. (47) at $s=\sigma+iT$ when integrated from $-M$ to $a$ tends to $o(1)$ as $T\to\infty$ and at $s=-M+iT$ when integrated from $-\infty$ to $\infty$, is less than $x^M$. The required result then follows.
 
 By similar reasoning, we can get the following theorems.
\begin{theorem}
Let $p>0$ and define
\begin{equation}
g\left( x \right) = \sum\limits_{v = 1}^\infty  {{e^{ - x{v^p}}}{v^{m - 1}}} \log v.    
\end{equation}
Then, as $x\to{0^{+}}$
\begin{equation*}
\sum\limits_{k = 1}^n {g\left( {kx} \right)}  \sim \left( {\frac{{\Gamma \left( {\frac{m}{p}} \right) - \Gamma \left( {\frac{m}{p}} \right)\log x}}{{{p^2}{x^{\frac{m}{p}}}}}} \right){\Delta _n}\left( {\frac{m}{p}} \right) + \sum\limits_{l = 0}^\infty  {\zeta '\left( {1 - m + pl} \right)} {\Delta _n}\left( { - l} \right)\frac{{{{\left( { - x} \right)}^l}}}{{l!}}   
\end{equation*}
\begin{equation}
+ \frac{{\zeta '\left( {1 - m + p} \right)}}{x}\left( {2 + {\rm O}\left( {{n^{ - 2}}} \right)} \right)     
\end{equation}
where $\frac{m}{p}\notin{Z^{-}}$. 
\end{theorem}

\end{proof}
\section{On some additional corollaries}
\textbf{5.1.} Ramanujan in his quarterly reports derived some additional corollaries of his master's theorem for which he did not gave any applications or example because the integrals in his corollaries contained infinite series and thus cannot be written in terms of any function that would suit the expansion in the integrand. In this section, we are going to apply Eqn. (15) and (16) to Ramanujan's derived corollaries so that they take more applicable form.

\textbf{5.2. } In his quarterly reports, Ramanujan records the following results [\cite{1}, Pg. 319, Corollary (vi)]
\begin{equation}
\int\limits_0^\infty  {\sum\limits_{k = 0}^\infty  {\frac{{\phi \left( k \right){{\left( { - x} \right)}^k}}}{{k!}}\sum\limits_{k = 0}^\infty  {\frac{{\psi \left( {2k} \right){{\left( { - {n^2}{x^2}} \right)}^k}}}{{\left( {2k} \right)!}}} } } dx = \sum\limits_{k = 0}^\infty  {\psi \left( {2k} \right)\phi \left( { - 2k - 1} \right){{\left( { - {n^2}} \right)}^k},}     
\end{equation}
\begin{equation}
\int\limits_0^\infty  {\sum\limits_{k = 0}^\infty  {\phi \left( {2k} \right){{\left( { - {x^2}} \right)}^k}\sum\limits_{k = 0}^\infty  {\frac{{\psi \left( {2k} \right){{\left( { - {n^2}{x^2}} \right)}^k}}}{{\left( {2k} \right)!}}} } } dx = \frac{\pi }{2}\sum\limits_{k = 0}^\infty  {\frac{{\phi \left( { - k - 1} \right)\psi \left( k \right){{\left( { - n} \right)}^k}}}{{k!}}.} 
\end{equation}
Using Eqn. (15) and (16), we can rewrite above two results as
\begin{equation}
\int\limits_0^\infty  {f\left( x \right)} \left[ {g\left( {nx} \right) + g\left( { - nx} \right)} \right]dx = 2\sum\limits_{k = 0}^\infty  {\psi \left( {2k} \right)\phi \left( { - 2k - 1} \right){{\left( { - {n^2}} \right)}^k}} ,    
\end{equation}
\begin{equation}
\int\limits_0^\infty  {\sum\limits_{k = 0}^\infty  {\phi \left( {2k} \right){{\left( { - {x^2}} \right)}^k}\left[ {g\left( {nx} \right) + g\left( { - nx} \right)} \right]} } dx = \frac{\pi }{2}\sum\limits_{k = 0}^\infty  {\frac{{\phi \left( { - k - 1} \right)\psi \left( k \right){{\left( { - n} \right)}^k}}}{{k!}}.}     
\end{equation}
where
\begin{equation}
g\left( x \right) = \sum\limits_{k = 0}^\infty  {\frac{{\psi \left( k \right){{\left( { - x} \right)}^k}}}{{k!}}}. \end{equation}

\textbf{5.3. }
Consider the following integral from [\cite{9} Pg. 490]
\begin{equation}
\int\limits_0^\infty  {{x^{s - 1}}{e^{ - ax}}\cos \left( {nx} \right)dx}  = \frac{{\Gamma \left( s \right)\cos \left\{ {{{\tan }^{ - 1}}\left( {\frac{n}{a}} \right)} \right\}}}{{{{\left( {{a^2} + {n^2}} \right)}^{\frac{s}{2}}}}}    
\end{equation}
where $s,a>0$. If we replace $x$ with $mx$ for $m>0$ we get
\begin{equation}
\int\limits_0^\infty  {{x^{s - 1}}{e^{ - amx}}\cos \left( {nmx} \right)dx}  = {m^{ - s}}\frac{{\Gamma \left( s \right)\cos \left\{ {{{\tan }^{ - 1}}\left( {\frac{n}{a}} \right)} \right\}}}{{{{\left( {{a^2} + {n^2}} \right)}^{\frac{s}{2}}}}}.    
\end{equation}
Now, if we applied the same method that we applied in the derivation of Theorem 2, we get the following general result for $s>0$ which resembles with Eqn. (55):
\begin{equation}
\int\limits_0^\infty  {{x^{s - 1}}f\left( {ax} \right)\left[ {f\left( {nx} \right) + f\left( { - nx} \right)} \right]dx}  = 2{\left( { - 1} \right)^{\frac{s}{2}}}\phi \left( 0 \right)\phi \left( { - s} \right)\Gamma \left( s \right)\frac{{\cos \left\{ {{{\tan }^{ - 1}}\left( {\frac{n}{a}} \right)} \right\}}}{{{{\left( {{a^2} + {n^2}} \right)}^{\frac{s}{2}}}}}.    
\end{equation}

\section{Conclusion}
In this work, we derived some extended versions of Ramanujan's master theorem. The type two interpolation formula was already derived by Ramanujan in his quarterly report but his method of derivation was different from the approach that we have employed. Some asymptotic expansions are also derived using the derived results. And finally, two of Ramanujan's corollaries of his master's theorem are rewritten and a analogue version is also presented. We believe that the results that we have derived in this paper can be use-full in the theory of Mellin transform and the method of brackets to evaluate some complicated indefinite integrals that appear in calculations.

\section{Declaration}
\textbf{Abbreviations:} No abbreviations\newline
\textbf{Author contributions:} Single author paper\newline
\textbf{Funding:} No funding received\newline
\textbf{Availability of data and materials:} My manuscript has no associated data.\newline
\textbf{Ethics approval and consent to participate:} Non applicable.\newline
\textbf{Consent for publication:} The author consent for publication.\newline
\textbf{Competing interests:} The author declare that they have no competing interests.\newline

\end{document}